\documentclass[a4paper,12pt]{amsart}
\usepackage{amsmath,amssymb}
\usepackage[shortlabels]{enumitem}
\usepackage{xcolor}

\def\B{\mathbb B}
\def\D{\mathbb D}
\def\C{\mathbb C}
\def\R{\mathbb R}
\def\N{\mathbb N}

\def\bcases{\begin{cases}}

\newcommand{\br}[1]{\left(#1\right)}
\def\ecases{\end{cases}}

\newcommand{\set}[1]{\left\{ #1\right\}}

\newtheorem{thm}{Theorem}
\newtheorem{prop}[thm]{Proposition}
\newtheorem{cor}[thm]{Corollary}
\newtheorem{lem}[thm]{Lemma}

\newtheorem{rem}[thm]{Remark}

\renewcommand{\Im}{\operatorname{\rm{Im}}}
\renewcommand{\Re}{\operatorname{\rm{Re}}}
\newcommand{\eps}{\varepsilon}
\renewcommand{\k}{\kappa}

\newenvironment{proof*}{\vskip 2mm\noindent {}}{\hfill $\Box$ \vskip 2mm}

\title{Quasi triangle inequality for the Lempert function}

\author{Nikolai Nikolov}
\address{N. Nikolov\\
Institute of Mathematics and Informatics\\
Bulgarian Academy of Sciences\\
Acad. G. Bonchev 8, 1113 Sofia, Bulgaria
\vspace{1mm}
\newline Faculty of Information Sciences\\
State University of Library Studies and Information Technologies\\
Shipchenski prohod 69A, 1574 Sofia,
Bulgaria}

\email{nik@math.bas.bg}

\author{Pascal J. Thomas}
\address{P.J. Thomas\\
Institut de Math\'ematiques de Toulouse; UMR5219 \\
Universit\'e de Toulouse; CNRS \\
UPS, F-31062 Toulouse Cedex 9, France}

\email{pascal.thomas@math.univ-toulouse.fr}

\thanks{The first named author was partially supported by the Bulgarian National
Science Fund, Ministry of Education and Science of Bulgaria under contract KP-06-N82/6. The
second named author wishes to thank the Institute of Mathematics and Informatics of the Bulgarian Academy
of Sciences for its hospitality during the time when most of this work was carried out.}


\begin{document}


\begin{abstract}
The (unbounded version of the) Lempert function $l_D$ on a domain $D\subset \C^d$
does not usually satisfy the triangle inequality, but on
bounded $\mathcal C^2$-smooth strictly pseudoconvex domains, it satisfies a quasi triangle inequality:
$l_D(a,c)\le C( l_D(a,b)+l_D(b,c))$. We show that pseudoconvexity is necessary for this property
as soon as $D$ has a $\mathcal C^1$-smooth boundary. 
We also give  estimates of the Lempert function and of other invariants in some
domains which are models for local situations, 
and derive some general local bounds
depending on the regularity of the boundary of a domain.
\end{abstract}

\maketitle

\section{Introduction}
\subsection{Invariant functions and triangle inequalities.}

The most commonly used invariant Finsler metric on the tangent space of a complex manifold $M$
is the Kobayashi-Royden metric, defined as follows using ``analytic discs'', i.e.
holomorphic maps from the unit disc $\D$ to $M$:
$$
\kappa_M(z;X):=\inf\{|\alpha| :\exists\varphi\in\mathcal O(\D,M), \varphi(0)=z, \alpha\varphi'(0)=X\}.
$$
It can be seen \cite{NP} as the infinitesimal version of the Lempert function,
$l_M(z,w):=\tanh^{-1}\ell_M(z,w)$, where
$$
\ell_M(z,w):=\inf\{|\alpha|:\exists\varphi\in\mathcal O(\D,M)
\hbox{ with }\varphi(0)=z,\varphi(\alpha)=w\}.
$$
Lempert \cite{L1}, \cite{L2} proved that $l_M$ (and thus $\ell_M$) satisfies the triangle inequality and defines a distance
when $M$ is a bounded convex domain in $\C^n$. This  extends to bounded $\mathcal C^2$-smooth $\C$-convex domains,
as a consequence of the theorem by Jacquet
\cite{Jq} that they can be exhausted by strongly linearly convex domains with $\mathcal C^\infty$ boundaries,
see also \cite[Remark 7.1.21 (b)]{JP}.

In general, as pointed out in \cite{L1}, $l_M$ does not satisfy the triangle inequality and
the largest pseudodistance not exceeding $l_M$ is $k_M:=\inf_{m\in\N}l_M^{(m)}= \lim_{m\to\infty}l_M^{(m)}$, where
\[
l_M^{(m)}(z,w) := \inf\{\sum_{j=0}^{m-1} l_M(y_j,y_{j+1}) , y_j \in M, y_0=z, y_m=w \} .
\]
We sometimes denote $l_M^{(\infty)}:=k_M$. 
We may define $\ell_M^{(m)}$ by substituting $\ell_M$
for $l_M$ in the above definition.
Note that $k_M \ge \inf_{m\in\N}\ell_M^{(m)}\ge \tanh k_M$ 
and those three functions  are pseudodistances.

The quantity $k_M$ is called the Kobayashi (pseudo)distance.
It turns out that $k_M$ is the integrated form of $\kappa_M,$ i.e.
\begin{equation}
\label{kobadef}
k_M(z,w)=\inf_{\gamma}\int_0^1\kappa_M(\gamma(t);\gamma'(t))dt,
\end{equation}
where the infimum is taken over all absolutely continuous curves
$\gamma:[0,1]\to M$ such that $\gamma(0)=z$ and $\gamma(1)=w.$
Good general references about those topics are the monographs \cite{JP} and \cite{Ko}.

From now on, we restrict attention to the case where $M$ is a domain $D$ in $\C^n.$
We say that a non-negative symmetric function $L$ on $D\times D$ satisfies a \emph{Quasi Triangle Inequality}
(QTI) if there exists $C\ge 1$ such that for any $x, y , z \in D$,
\begin{equation}
\label{wti}
L (x,z) \le C \left( L (x,y) + L (y,z) \right).
\end{equation}
Notice that when $l_D$ satisfies a QTI (resp. the triangle inequality), then $\ell_D$
satisfies the corresponding inequality.

There are infinitesimal analogues for this.
We say that a Finsler pseudometric $F$
satisfies the triangle inequality when $F(z; X+Y) \le F(z;X)+F(z;Y)$
for any  $z\in M$ and $X,Y\in T_z M$,
and a \emph{quasi triangle inequality} if there exists
$C>0$ such that 
\[
F(z; X+Y) \le C(F(z;X)+F(z;Y))
\]
for any  $z\in M$ and $X,Y\in T_z M$.

In general, the Kobayashi-Royden infinitesimal pseudometric $\kappa_M$
does not satisfy the triangle inequality; the Kobayashi-Buseman
 infinitesimal pseudometric, given by
 $\hat \kappa_M (z;X):= \inf \kappa^{(m)}_M (z;X)$, where
\[
\kappa^{(m)}_M (z;X):= \inf\left\{ \sum_{j=1}^m \kappa_M (z;X_j): 
X= \sum_{j=1}^m X_j \right\}
\]
was devised to supplement this. It turns out that 
$k_M$ is also the integrated form of $\hat \kappa_M$.

\medskip
Unlike the actual triangle inequality, a QTI is satisfied by $\ell_D$ for large classes of domains.

\begin{prop}
\label{suff}
If $D \subset \C^n$ is a bounded $\mathcal C^2$-smooth strictly pseudoconvex domain, then $l_D$,
and therefore $\ell_D$, satisfy a QTI.
\end{prop}
This is a corollary of the main result in \cite{N}: in a bounded $\mathcal C^2$-smooth strictly pseudoconvex domain,
there is a constant $A \ge 1$ such that $k_D \le l_D \le A c_D \le A k_D$, where $c_D$ stands for the
Carath\'eodory distance. It is however possible to prove the conclusion of Proposition \ref{suff} for $\ell_D$ in a more elementary way, using
a natural localization lemma for the Lempert function which will be useful in the rest of the paper (Lemma \ref{royfin}),
and we do so in Section \ref{suffcond}.

There are many examples of domains where $l_D$ satisfies a QTI without satisfying the triangle inequality.
For instance, let $D$ be a pseudoconvex balanced domain. It is known that the
equality $l_D(0,\cdot)=k_D(0,\cdot)$ is equivalent to the convexity of $D$ (see e.g. \cite[Proposition 4.3.10 (b)]{JP}).
Take now $D$ to be non-convex. If $l_D$ verified the triangle inequality,
then it would be equal to $k_D$ everywhere. So one may find $\eps>0$ and $x_1,x_2,x_3\in D$
such that $l_D(x_1,x_2)+l_D(x_2,x_3)<l_D(x_1,x_3)-\eps.$ Since $D$ is pseudoconvex, we may exhaust it by strictly
pseudoconvex domains $(D_j)_{j\ge 1}.$ Then $l_{D_j}(x_k,x_m)\to l_D(x_k,x_m),$ $1\le k<m\le 3,$ (see e.g.
\cite[Proposition 3.3.5 (a)]{JP}) and hence $l_{D_j}(x_1,x_2)+l_{D_j}(x_2,x_3)<l_{D_j}(x_1,x_3)$ for
any $j$ large enough. On the other hand, as we have already claimed, $D_j$ satisfies a QTI.

\subsection{Main result.}

\begin{thm}
\label{pcnc}
If $D \subset \C^n$ is a bounded $\mathcal C^1$-smooth domain, and $\ell_D$ or $\kappa_D$ satisfies a QTI,
then $D$ is pseudoconvex.
\end{thm}

The scheme of the proof of Theorem \ref{pcnc} is as follows: if the domain is not pseudoconvex, we can find a point $p\in \partial D$
which is, in a certain sense, strictly pseudoconcave.  Near such a point we can find pairs of points $z,w \in D$
such that $\ell_D^{(2)}(z,w)  \ll\ell_D(z,w)$ as $z,w \to p$. This is done in Section \ref{neccond} by comparing
upper estimates for $\ell_D^{(2)}$ near a pseudoconcave point (in
Proposition \ref{upbd}, which does not depend on the smoothness
of $\partial D$) and lower estimates for $l_D$ near any $p$, 
when $\partial D$ is merely
$\mathcal C^{1}$-smooth (Proposition \ref{C1}).

To obtain the lower estimates, we study some model domains in $\C^2$. We thought it interesting to study more
completely the growth of the Lempert function along the normal to the boundary in those domains,
and this is done in Section \ref{model}. In Section \ref{more}, 
we study estimates for another invariant, the Sibony metric, and
derive some further consequences of the estimates for model domains.

When the Lempert function $\ell_D$ satisfies the triangle inequality,
it is clearly continuous (and thus $l_D$ as well). The following  shows that this is still the case when only a QTI is satisfied, and that it extends
to some other invariants.

\begin{cor}
\label{corcont}
If $D \subset \C^n$ is a bounded $\mathcal C^1$-smooth  domain, and $\ell_D$ (resp. $\k_D$) satisfies a QTI,
then for all $m \in\N$, $\ell_D^{(m)}$ and $l_D^{(m)}$ 
(resp. $\kappa_D^{(m)}$ and $\hat\kappa_D$) are continuous functions.
\end{cor}

This corollary follows from Theorem \ref{pcnc} because any bounded pseudoconvex domain $D$ with Lipschitz
boundary is hyperconvex, thus taut (see e.g. \cite[Remark 3.2.3 (b)]{JP}; $D$ is called taut
if the space $\mathcal O(\D,D)$ is normal), and from the following proposition,
the proof of which is given in Section \ref{appen}.

\begin{prop}
\label{cont}
(a) For any domain $D\subset\C^n,$ $l_D^{(m)},$ $\ell_D^{(m)},$ $\k_D^{(m)},$ and $\hat\k_D$ are upper semicontinuous functions uniformly
in $m\in\N$ pointwise.

(b) For any taut domain $D\subset\C^n$ and any $m\in\N,$ the infimum in the definitions of $l_D^{(m)}$, $\ell_D^{(m)}$  and $\k_D^{(m)}$
is attained (i.e. there exist $m-1$ intermediate
points, resp. $m$ vectors,  and corresponding analytic discs realizing the
minimum), and these functions, as well as $\hat\k_D,$ are continuous.
\end{prop}

\subsection{Open Questions.}
\begin{enumerate}
\item
The above shows that continuity $l_D$ the Lempert function can be seen as another relaxation of
the condition that $\ell_D$ satisfy the triangle inequality. For a $\mathcal C^1$-smooth bounded
domain $D$, does the continuity of $l_D$ imply pseudoconvexity of $D$? It would mean that among bounded
$\mathcal C^1$-smooth domains, continuity of the Lempert function is equivalent
to pseudoconvexity (or tautness, or hyperconvexity, since they are all equivalent in this case).
\item
Our methods only work for domains with boundary of at least $\mathcal C^1$ regularity. Does the
implication still hold for domains with Lipschitz boundary, or even H\"older?

If the boundary is very irregular, there are counterexamples.  For instance, let $D=\B^2 \setminus \{0\}$,
where $\B^2$ is the unit ball in $\C^2$. Then one can easily show  that
$\ell_D=\ell_{\B^2}|_D$ and thus satisfies the triangle inequality, while $D$ is clearly not pseudoconvex.
In Remark \ref{avoid}, Section \ref{model}, we mention some examples of the form $D \setminus E$, where $E$ is an ``exceptional'' closed set;
but none of those examples has a boundary locally representable by the graph of a continuous function.
\item
Finally, the main open question is whether the converse of our theorem holds for $\mathcal C^1$-smooth
bounded domains: if $D$ is pseudoconvex, does $\ell_D$ satisfy a QTI?

\end{enumerate}

\section{Sufficient conditions for a QTI}
\label{suffcond}

We first give a Lempert function version of a localization Lemma which goes back to Royden \cite{R}
for the case of the Kobayashi-Royden metric. For $A,B \subset D$, let
$\ell_D(A,B):= \inf\{ \ell_D(z,w): z\in A, w\in B\}$.

\begin{lem}
\label{royfin}
Let $V \Subset U$ be two open sets such that $D\cap U$ is nonempty and connected, $D\setminus U \neq \emptyset$;
let $z,w \in D\cap V$,
then
\[
\ell_D (D\cap V, D\setminus U)\ell_{D\cap U}(z,w) \le \ell_D(z,w) .
\]
\end{lem}
\begin{proof}
Let $\eps>0$, and $\varphi$ be any holomorphic map from $\D$ to $D$ such that $\varphi(0)=z$, and
$\varphi(\alpha)=w$ for some $\alpha < \ell_D(z,w)+\eps$. If for any $\eps>0$ and any $r<1$, $\varphi (D(0,r))\subset D\cap U$,
then $\ell_D(z,w) =\ell_{D\cap U}(z,w) $ and we are done.  If not, let
\[
r_0:= \sup\{ r: \varphi(D(0,r)) \subset D\cap U\} <1.
\]
Since $\varphi(0)=z\in V$, we have $r_0 \ge \ell_D (D\cap V, D\setminus U)$.
Let $\psi(\zeta):= \varphi(r_0\zeta)$ for $\zeta \in \D$;
$\psi(\D)\subset D\cap U$, so $\frac\alpha{r_0} \ge\ell_{D\cap U}(z,w)$, thus
\[
 \ell_D(z,w)+\eps > \alpha \ge\ell_{D\cap U}(z,w) \ell_D (D\cap V, D\setminus U).
 \]
Since this holds for all $\eps>0$, we are done.
\end{proof}

\begin{proof*}{\it Proof of Proposition \ref{suff} for $\ell_D$.}
Suppose the result fails; then there exists for any $k\in\N$ points $a_k, b_k, c_k \in D$
such that
\begin{equation}
\label{trifail}
1 \ge \ell_D(a_k, c_k) \ge k\left( \ell_D(a_k, b_k) + \ell_D( b_k, c_k) \right) \ge
k(\mbox{diam }D)^{-1}\left( \| a_k - b_k \| + \| b_k- c_k\| \right),
\end{equation}
so passing to a subsequence all three points $a_k, b_k, c_k$ tend to the same $p\in \overline D$.
If $p\in D$, then in a neighborhood of $p$, $l_D (z,w) \asymp \|z-w\|$, and \eqref{trifail} cannot hold
when $k$ is large, so $p \in \partial D$.

Since $D$ is strictly pseudoconvex, we may take neighborhoods $V \Subset U$ of $p$ such that $D\cap U$ is biholomorphic to a convex domain. For $k$ large enough, $a_k, b_k, c_k \in D\cap V$. By Lemma \ref{royfin}, $l_D \asymp l_{D\cap U}$ on $D\cap V$,
which contradicts \eqref{trifail} once again for possibly larger $k$.
\end{proof*}

\section{Necessary condition for a QTI}
\label{neccond}

\begin{proof*}{\it Proof of Theorem \ref{pcnc}.}
We proceed by contradiction: supposing that $D$ is not pseudoconvex, we will show that there are triples of points
which violate any QTI.

Recall that if $D$ is not pseudoconvex, then there exists an affine $2$-complex-dimensional subspace $P$ such that
$D\cap P$ is not pseudoconvex as a subset of $\C^2$ \cite{Hi}, see also \cite{Jb}. Choose coordinates
such that $D\cap P=D_0 \times\{0\}$, with $D_0 \subset \C^2$.

First we reduce part of our problem to the study of a model domain.

\begin{lem}
\label{betterHo}
If $D\subset \C^2$ is a bounded non pseudoconvex domain, there exists $p\in \partial D$, $U$ a neighborhood of $p$,
and $\Phi$ a biholomorphism on $U$ such that $\Phi(p)=0$,
$\Phi^{-1}(\overline {\D^2}) \subset U$ and 
 $\Phi^{-1}(G_2)  \subset D$, where 
 $G_2:= \{ z \in \D^2: \Re z_1 < |z_2|^2\}$.
\end{lem}

\begin{proof}
By \cite[Theorem 4.1.25]{Ho}, since $D$ is not pseudoconvex, there exists $p\in \partial D$, a neighorhood $U_1$ of $p$, and
an affine coordinate system in which $p=0$ and $D\cap U_1 \supset \D^2 \cap \{\Re z_1 + q(z_1,z_2)<0\}$, where $q$ is a
real valued  polynomial, homogeneous of degree $2$, with $\frac{\partial^2 q}{\partial z_2 \partial \bar z_2}(0)=-1$.

We can write $q(z_1,z_2)= q_1(z_1) + q_2(z_1,z_2) + \Re(\alpha_2 z_2^2) - |z_2|^2$, where the $q_j$ 
are real-valued and homogeneous of degree $2$, $\alpha_2 \in \C$, and 
\[
|q_2(z_1,z_2)|\le C_2 |z_1||z_2| \le \frac{C_2^2}2 |z_1|^2 + \frac12 |z_2|^2,
\]
so that, in a small enough neighborhood $U_2$ of $0$,
\[
q(z_1,z_2)\le C_1|z_1|^2 + \frac{C_2^2}2 |z_1|^2  + \Re(\alpha_2 z_2^2) - \frac12 |z_2|^2 .
\]
Now define a local biholomorphism $\Phi$ by 
\[
(z_1,z_2) = \Phi^{-1}(z'_1, z'_2):= (z_1'+\alpha_1 {z'_1}^2 -\alpha_2 {z'_2}^2, z'_2),
\] 
where $\alpha_1> C_1+\frac{C_2^2}2$. 
In these new coordinates, the condition $\Re z_1 + q(z_1,z_2)<0$ is equivalent to 
$\Re z'_1 + \tilde q(z'_1,z'_2)<0$ with
\begin{multline*}
\tilde q(z'_1,z'_2)< (C_1+\frac{C_2^2}2 ) |z'_1|^2 + o(|z'_1|^2+|z'_2|^2) + \alpha_1 \Re ({z'_1}^2)  - \frac12 |z'_2|^2
\\
\le (C_1+\frac{C_2^2}2 +\alpha_1 +1)(\Re z'_1)^2 - \frac13 |z'_2|^2,
\end{multline*}
when $(z'_1, z'_2)\in U_2$, a small enough neighborhood of $0$. Finally,
setting $C_3:=C_1+\frac{C_2^2}2 +\alpha_1 +1$, the condition reduces to
$\Re z_1 + C_3 (\Re z'_1)^2 <\frac13 |z'_2|^2$, i.e. 
$\Re z_1 <\frac13   (1+ C_3 (\Re z'_1))^{-1} |z'_2|^2$, which is implied by
$\Re z_1 <\frac14 |z'_2|^2$ when $z\in U_3$, a possibly smaller neighborhood of $0$.

A last linear change of coordinates yields the desired form of $G_2$.
\end{proof}

\begin{prop}
\label{upbd}
Choose coordinates as in Lemma \ref{betterHo} such that $0\in 
\partial D$ with $G_2\subset D$,
for $t \in (0,1)$, let $p_t := (-t,0)$. 

Then for $0< \delta < \eps$ small enough,
\[
\ell_{D}^{(2)}(p_\delta, p_\eps) \le\ell_{G_2}^{(2)}(p_\delta, p_\eps) \le 2 \frac{\eps-\delta}{\eps^{1/2}},
\]
and $\kappa^{(2)}_D (p_\eps; (1,0)) \le 2 \eps^{1/2}$.
\end{prop}
\begin{proof}
Let $q_\delta:= (-\delta, \frac{\eps-\delta}{\eps^{1/2}})$. Considering the disc $\psi (\zeta)= (-\delta,\zeta)$,
we clearly have $\ell_{G_2}(p_\delta,q_\delta) \le \frac{\eps-\delta}{\eps^{1/2}}$.

On the other hand, if we set $\zeta=x+iy$, the disc $\varphi(\zeta):= (-\eps+\eps^{1/2}\zeta, \zeta)$  satisfies
\[
|\varphi_2(\zeta)|^2 - \Re \varphi_1(\zeta) =
x^2+y^2 +\eps -\eps^{1/2} x \ge (x-\frac12 \eps^{1/2})^2 + \frac{3\eps}4 >0,
\]
so for $\eps$ small enough, $\varphi(\D)\subset G_2$ and since $\varphi (0)=p_\eps$,
$\varphi (\frac{\eps-\delta}{\eps^{1/2}})= q_\delta$,
$\ell_{G_2}(p_\eps,p_\delta) \le \frac{\eps-\delta}{\eps^{1/2}}$.

To obtain the estimate on  $\kappa^{(2)}_{G_2} (p_\eps; (1,0))$, notice
that $(1,0)= (1, \eps^{-1/2})+(0, -\eps^{-1/2})$. Then using 
$\varphi'(0)=\eps^{1/2}(1, \eps^{-1/2})$,
we see that $\kappa_{G_2} (p_\eps; (1, \eps^{-1/2})) \le \eps^{-1/2}$,
and likewise using the map $\psi (\zeta)= (-\eps,\zeta)$, we obtain
$\kappa_{G_2} (p_\eps; (0, -\eps^{-1/2})) \le \eps^{-1/2}$.
\end{proof}

To obtain the estimate in $D\subset \C^n$, we also denote $p_t=(-t,0, \dots, 0) \in \C^n$ by a slight abuse, then with $D_0$
as in the beginning of this section, 
$\ell_{D}^{(2)}(p_\delta, p_\eps) \le\ell_{D_0}^{(2)}(p_\delta, p_\eps)\le 2 \frac{\eps-\delta}{\eps^{1/2}}$, by 
Proposition \ref{upbd}. The computation for  $\kappa^{(2)}$ is analogous.


The proof of the theorem will conclude with the following lower estimates on $\ell_D(p_\delta,p_\eps)$ and $\kappa_D (p_\eps; (1,0))$.
\end{proof*}

\begin{prop}
\label{C1}
Let $D \subset \C^n$ be a bounded $\mathcal C^{1}$-smooth  domain, and
for any $p\in \partial D$, let $p_t := p+t n_p$, where $n_p$ is the inner normal vector at $p$. Then for
$0< \delta < \eps$ small enough, $\ell_D(p_\delta,p_\eps)) \gg \frac{\eps-\delta}{\eps^{1/2}}$ and $\kappa_D (p_\eps; (1,0))\gg \eps^{-1/2}$
as $\eps\to0$.
\end{prop}

\begin{proof}
Using Lemma \ref{royfin}, it will be enough to prove the conclusion of the Proposition for  $D \cap \D^n$,
so by the contracting property of the Lempert function, 
it is enough to find a domain $G$ such that  $D \cap \D^n \subset G \subset \D^n$
so that the conclusion holds for $G$.

Represent the boundary locally by $\{ \Re z_1 < \Phi (\Im z_1, z')\}$,
with $\nabla \Phi(0,0)=0$.
The regularity of the boundary  implies that
$\Phi (\Im z_1, z') \le \psi(|\Im z_1|+|z'|)$, with $\psi: [0,\infty) \longrightarrow [0,\infty)$,
$\psi(x)=x\psi_1(x)$, $\lim_{x\to0}\psi_1(x)=0$. We way take $\psi_1$ increasing.
Localizing, we reduce ourselves to studying $l_G(p_\eps,p_\delta)$,
where $G:= \{ z \in \D^n: \Re z_1 < \psi(|\Im z_1|+\|z'\|) \}$.

Let $\varphi=(\varphi_1, \dots, \varphi_n): \D \longrightarrow D$ be 
a holomorphic map such that $\varphi (0)=  p_\eps$,
and there is $\alpha>0$ such that $\varphi(\alpha)= p_\delta$. 
Write $B_\alpha (\zeta):= \zeta \frac{\zeta-\alpha}{1-\alpha \zeta}$.
For $j\ge 2$, $\varphi_j(\zeta) = h_j(\zeta) B_\alpha (\zeta)$,
 where $h_j$ is a holomorphic function with $\sup_{\D} |h_j| \le 1$.

It follows that $\|(\varphi_2(\zeta),\dots,\varphi_n(\zeta))\|\le C \beta^2$ when $|\zeta|<\beta$,
so
\[
\varphi_1(D(0,\beta)) \subset
\{ z\in \D: \Re z < \psi(|\Im z|+C\beta^2)\}=:P_\beta.
\]

Choose $\beta_\eps$ to be the unique number such that
$\psi(C \beta_\eps^2)=\eps$.  Notice that $\beta_\eps = \eps^{1/2} F(\eps)$, with $\lim_{\eps\to0} F(\eps)=\infty$. If $\alpha\ge \beta_\eps$, we are done.

If not, define the holomorphic map $f$ on $\C\setminus  [\eps,\infty) \supset P_{\beta_\eps}$
by $f(\xi):=-(-\xi+\eps)^{1/2}$,
where the argument of the square root is taken in $(-\pi,\pi)$.
We then have $f(P_{\beta_\eps})\subset \{\Re z <0\}$, and 
$f\circ \varphi_1 (0) = f(-\eps)= - (2\eps)^{1/2}$,
$f\circ \varphi_1 (\alpha) = f(-\delta)= - (\eps+\delta)^{1/2}$,
so that
\[
\frac{\alpha}{\beta_\eps} \ge
\frac{(2\eps)^{1/2}-(\eps+\delta)^{1/2}}{(2\eps)^{1/2}+(\eps+\delta)^{1/2}}
=
\frac{\eps-\delta}{((2\eps)^{1/2}+(\eps+\delta)^{1/2})^2}
\ge
\frac{\eps-\delta}{8\eps} ,
\]
so $\alpha \ge \beta_\eps \frac{\eps-\delta}{8\eps} = \frac18 F(\eps)\frac{\eps-\delta}{\eps^{1/2}}$.

\smallskip

For the infinitesimal analogue, suppose now that $\varphi(0)=p_\eps$
and there is $\lambda \in \C$ such that $\varphi'(0)=\lambda (1,0,\dots,0)$.
Notice that $B_0(\zeta)=\zeta^2$, so we have the same inequalities on
$\varphi_j$, $j\ge 2$. Choose $\beta_\eps$, $P_{\beta_\eps}$ and $f$ as above,
and applying the Schwarz Lemma to the map
\[
\D \longrightarrow \D, \qquad 
z\mapsto \frac{f\circ \varphi_1 (\beta_\eps z) + (2\eps)^{1/2}}{f\circ \varphi_1 (\beta_\eps z) - (2\eps)^{1/2}},
\]
we find $1 \ge \beta_\eps |\lambda| \frac{1}{(2\eps)^{1/2}} \frac{1}{(2\eps)^{1/2}}= |\lambda| \eps^{-1/2} F(\eps)$, so that $\kappa_D (p_\eps; (1,0))\ge  \eps^{-1/2} F(\eps)$.
\end{proof}

\section{Model Domains}
\label{model}

For $0<\mu \le 2$, let us define
\begin{equation}
\label{defG}
G_\mu := \{z\in \D^n: \Re z_1 < \sum_{j=2}^n |z_j|^\mu \} \mbox{ and } 
\tilde G_\mu := \{z\in \D^n: \Re z_1 < | \Im z_1|^\mu + \sum_{j=2}^n |z_j|^\mu \}.
\end{equation}

We have used the case $n=2, \mu=2$ above.

Note that since $G_\mu \subset \tilde G_\mu $, for any invariant metric or function $L$, $L_{G_\mu} \ge L_{\tilde G_\mu}$, so
we will prove the upper estimates for $G_\mu$ and the lower estimates for $\tilde G_\mu$. We also
have $G_\mu \subset G_{\mu'} $, $\tilde G_\mu \subset \tilde G_{\mu'} $ when ${\mu} \ge {\mu'}$,
so $\mu \mapsto \ell_{\tilde G_\mu}(p_\delta, p_\eps)$, $\mu \mapsto \ell_{G_\mu}(p_\delta, p_\eps)$
are increasing functions of $\mu$. Recall also that 
$l/\ell \to 1$ as $l\to 0$, so for estimates near $0$ it does not matter which version of the Lempert function we are using.

Infinitesimal metrics for points and vectors along the inner normal to the origin
for those models (and for more general pseudoconcave domains) have already been studied in \cite{DNT, FL, Fu1, Fu2} among others.

In what follows, for $\zeta, \eta \in \D$, we denote by $m_{\D} (\zeta, \eta):=\left| \frac{\zeta-\eta}{1-\zeta \bar \eta} \right|$
the pseudohyperbolic distance (or M\"obius function) in the unit disc. Also we use the notation $X_+:= \max(X,0)$, for $X\in \R$.

\begin{prop}
\label{modelest}

For $t \in (0,1)$, let $p_t := (-t,0)$. Let $0< \delta < \eps \le \frac12$, $0 < \mu \le 2$, and 
 $2 \le m \le \infty$.
\begin{multline*}
\ell_{G_\mu}(p_\delta, p_\eps) \asymp \ell_{\tilde G_\mu}(p_\delta, p_\eps) \asymp \frac{(\eps-\delta)}{\eps^{(1-\frac1{2\mu})_+}} 
\ge
\frac{(\eps-\delta)}{\eps^{(1-\frac1{\mu})_+}} \asymp \ell^{(m)}_{G_\mu}(p_\delta, p_\eps) \asymp \ell^{(m)}_{\tilde G_\mu}(p_\delta, p_\eps).
\end{multline*}
\end{prop}
Note that the points $(z_1,z_2)\in \partial G_\mu$
or $\partial \tilde G_\mu$ with $0<|z_2|<1$, $\Im z_1 \neq 0$, are all
$\mathcal C^\infty$-smooth
strictly pseudoconcave points, so that even when the (quasi)
triangle inequality is satisfied for points of the form $p_t$,
which tend to $(0,0)$,
we can find counterexamples to the inequality near any
of those other boundary points, by Theorem \ref{pcnc}.

\begin{proof}
In all cases, the projection to the first coordinate maps $\tilde G_\mu$
to $\D$, so for all $m\in \N^*$, $\ell_{\tilde G_\mu}^{(m)}(p_\delta, p_\eps)
\ge k_{\D}(-\delta, -\eps) 
=m_{\D}(-\delta, -\eps) \asymp \eps-\delta$.
\smallskip

{\it The case when $0 < \mu \le \frac12$.}

Note that in this case all the estimates reduce to $\eps-\delta$.
It will be enough to exhibit a map $\varphi:\D\longrightarrow G_{\frac{1}{2}}$ such that $\varphi(-\eps)=p_\eps$,
$\varphi(-\delta)=p_\delta$. Let $\varphi(\zeta):=\left(\zeta,
\frac{\zeta+\eps }{1+\zeta\eps} \frac{\zeta+\delta}{1+\zeta\delta}, 0, \dots, 0\right)$.
Let $\zeta=x+iy$, $x,y \in \R$.

It is elementary to show that $\left| \frac{\zeta+\eps }{1+\zeta\eps} \right| \ge \left| \frac{x+\eps }{1+x\eps} \right| $. So to
prove that $\varphi$ maps the unit disc to $G_{\frac{1}{2}}$, it is
enough to show that for $0\le x < 1$, we have
\[
x^2 < \frac{x+\eps }{1+x\eps} \frac{x+\delta }{1+x\delta} ,
\]
which reduces to $(1-x^2)(\delta \eps (x^2+1) + (\delta +\eps) x) >0$,
which completes the proof.
\smallskip

{\it Upper bound for $\ell_{G_\mu}$.}

Consider the map $\varphi:\D\longrightarrow \D^n$ given by
\begin{equation}
\label{firstdegmap}
\varphi(\zeta):=\left(-\eps + \eps^{1-\frac1{2\mu}}\zeta, B_\alpha(\zeta), 0, \dots, 0\right),
\end{equation}
with $\alpha:= \frac{(\eps-\delta)}{\eps^{1-\frac1{2\mu}}} $, and $B_\alpha(\zeta):= \zeta \frac{\zeta-\alpha}{1-\alpha\zeta}$,
a finite Blaschke product with zeros at $0$ and $\alpha$. Thus $\varphi(0)= p_\eps$, $\varphi(\alpha)=p_\delta$, so 
if we can prove that $\varphi(\D) \subset G_\mu$, then
$\ell_{G_\mu}(p_\delta, p_\eps) \le \alpha$, and this finishes the proof.

Let $\zeta=x+iy$, $x,y\in\R$. It is elementary to check that $|\varphi_2(x+iy)|\ge |\varphi_2(x)|$, so to check
that $\varphi(\D) \subset G_\mu$, it is enough to prove that for $x\in (-1,+1)$, 
\[
-\eps + \eps^{1-\frac1{2\mu}}x < \left(x  \frac{x-\alpha}{1-\alpha x}\right)^\mu \mbox{ i.e. }
-1 + \frac{x}{\eps^{\frac1{2\mu}}} < \left(  \frac{x}{\eps^{\frac1{2\mu}}} \left(\frac{x}{\eps^{\frac1{2\mu}}} - \frac{\alpha}{\eps^{\frac1{2\mu}}} \right) \right)^\mu \frac1{(1-\alpha x)^\mu}.
\] 
The inequality is obvious for $x\le \eps^{\frac1{2\mu}}$, so we only consider $x> \eps^{\frac1{2\mu}}$.
Then the last factor is larger than $1$, 
so setting $t:= \frac{x}{\eps^{\frac1{2\mu}}} \in (1,\infty) $, 
and noticing that 
$\frac{\alpha}{\eps^{\frac1{2\mu}}} = \frac{\eps-\delta}{\eps}$, 
it is enough to prove that 
$t-1< t^\mu(t-\frac{\eps-\delta}{\eps})^\mu$ for $t>1$.
\vskip.3cm

{\it Case 1: $\frac12 < \mu \le 1$.}

Set as new auxiliary variable $\tau:= (t-1)^{\frac1\mu}$. We need to see that 
\begin{equation}
\label{ineqtau}
\tau\le (1+\tau^\mu)\left(\frac\delta\eps+\tau^\mu\right)=\frac\delta\eps+ \left(\frac\delta\eps+1\right) \tau^\mu+ \tau^{2\mu}. 
\end{equation}
For $\tau\le1$, $ \left(\frac\delta\eps +1\right) \tau^\mu \ge  \tau^\mu \ge \tau$,
and for $\tau\ge 1$, $\tau^{2\mu}\ge \tau$.

\vskip.3cm

{\it Case 2: $1 < \mu $.}

\vskip.1cm

{\it Case 2.1: $(1-\frac1{2\mu}) \eps < \delta < \eps$.}

To obtain \eqref{ineqtau}, it is now enough to prove
\begin{equation}
\label{ineqtau2}
\tau\le (1+\tau^\mu)\left(1-\frac1{2\mu}+\tau^\mu\right), \mbox{ for } \tau >0.
\end{equation}
The inequality is trivially verified for $\tau \le 1-\frac1{2\mu}$ (and for $\tau \ge 1$). The derivative of the righthand side is 
equal to 
\[
2\mu \tau^{\mu-1} \left( 1-\frac1{4\mu} + \tau^\mu\right) \ge
2\mu \left(1-\frac1{2\mu}\right)^{\mu-1} \left( 1-\frac1{4\mu} + \left(1-\frac1{2\mu}\right)^\mu\right) ,
\]
for $\tau \ge 1-\frac1{2\mu}$. Using $\log (1-t) \ge - (2\ln 2)t$ for $0\le t \le \frac12$, 
we find $ \left(1-\frac1{2\mu}\right)^\mu\ge \frac12$, and finally the righthand side of the above equation is
$\ge \frac32 \mu - \frac14 > 1$, so that \eqref{ineqtau2} holds on $(0,\infty)$.


{\it Case 2.2: $0 < \delta \le (1-\frac1{2\mu}) \eps$.}

Here we cannot use the map defined in \eqref{firstdegmap} any longer. The map will depend  $\nu:=\lceil \mu \rceil:=
\min\{m\in \N: m\ge \mu\}$.

Define $\varphi: \D \longrightarrow \D^2$  by
\begin{equation}
\label{seconddegmap}
\varphi(\zeta):= (-\delta-a\zeta^\nu, B_\alpha(\zeta))\mbox{, where }a=a_0 \eps^{1-\frac\nu{2\mu}}, 
\end{equation}
for $a_0$ small enough, to be chosen depending on $\mu$, 
 $\alpha:= \left(\frac{\eps-\delta}a \right)^{1/\nu}$. 
Then for $\eps$ small enough we do have $\varphi(\D)\subset \D^2$,
and $\varphi(0)= p_\delta$, $\varphi(\alpha)= p_\eps$, and 
\[
\left(\frac{\eps}{2\mu a} \right)^{1/\nu} = 
\frac1{(2\mu)^{1/\nu}}  a_0^{-1/\nu} \eps^{1/2\mu}
\le
\alpha \le   a_0^{-1/\nu} \eps^{1/2\mu}.
\]

If we can prove that $\varphi(\D)\subset G_\mu$, then $\ell_{ G_\mu}(p_\delta, p_\eps) \le \alpha$ and we are done.

To see that $\Re \varphi_1(\zeta) < |\varphi_2(\zeta)|^\mu$ for $\zeta \in \D$, 
it is enough to do it when $\Re \zeta^\nu \ge 0$, which implies that 
$|\cos (\arg \zeta)| \le \cos \frac\pi\nu$, 
so $|\Im \zeta| \ge |\zeta| \sin \frac\pi\nu$.
In particular, 
$|\zeta-\alpha|\ge c_1 \max (a_0^{-1/\nu} \eps^{1/2\mu}, |\zeta|)$,
with $c_1$ depending only on $\mu$. 

Then it is enough to show that, for $|\zeta|<1$,
\begin{equation}
\label{rebla}
a_0 \eps^{1-\frac\nu{2\mu}} |\zeta|^\nu 
\le \frac{c_1^\mu}{2^\mu} |\zeta|^\mu \max (a_0^{-1/\nu} \eps^{1/2\mu}, |\zeta|)^\mu.
\end{equation}
When $|\zeta| \ge a_0^{-1/\nu} \eps^{1/2\mu}$, this reduces to
\[
a_0 \eps^{1-\frac\nu{2\mu}} \le \frac{c_1^\mu}{2^\mu} |\zeta|^{2\mu-\nu},
\]
which it is enough to check for $|\zeta| = a_0^{-1/\nu} \eps^{1/2\mu}$
since $2\mu-\nu>0$. For this value of $|\zeta|$, it then reduces to 
$a_0^{2\mu/\nu}\le \frac{c_1^\mu}{2^\mu}$. 

When on the other hand $|\zeta| \le a_0^{-1/\nu} \eps^{1/2\mu}$, 
\eqref{rebla} reduces to 
\[
a_0 \eps^{1-\frac\nu{2\mu}} |\zeta|^{\nu-\mu} \le \frac{c_1^\mu}{2^\mu} 
 a_0^{-\mu/\nu} \eps^{1/2},
\]
which it is enough to check for $|\zeta| = a_0^{-1/\nu} \eps^{1/2\mu}$
since $\nu-\mu\ge 0$. But this leads to the same inequality as above, and 
we are done.
\vskip1cm

\smallskip

{\it Lower bound for $\ell_{\tilde G_\mu}$ when $\frac12 < \mu$.}


Let $\varphi=(\varphi_1, \dots, \varphi_n): \D \longrightarrow G_\mu$ be a holomorphic map such that $\varphi (0)=  p_\eps$,
and there is $\alpha>0$ such that $\varphi(\alpha)= p_\delta$. Write $B_\alpha (\zeta):= \zeta \frac{\zeta-\alpha}{1-\alpha \zeta}$.
For $j\ge 2$, $\varphi_j(\zeta) = h_j(\zeta) B_\alpha (\zeta)$ where $h_j$ is a holomorphic function with $\sup_{\D} |h_j| \le 1$.

Choose  $\beta$ such that $(2\beta^2)^\mu=  \eps$.
We may assume $\alpha <\beta$,
since the conclusion holds otherwise. For $j\ge 2$ and $|\zeta|\le \beta$,
$|\varphi_j(\zeta)| \le \beta \frac{\beta+\alpha}{1+\beta \alpha}
\le 2 \beta^2$. This implies that $\varphi_1 (D(0,\beta))
\subset \{ z \in \C: \Re z < C_\mu (\eps+|\Im z_1|^{\mu})\}=:P$.

As in the proof of Proposition \ref{C1}, define a map $f$ on $\C\setminus  [\eps,\infty) \supset P$
by $f(\xi):=-(-\xi+\eps)^{1/2}$,
where the argument of the square root is taken in $(-\pi,\pi)$.
We then have $f(P)\subset \{\Re z <0\}$,
$f\circ \varphi_1 (0) = f(-\eps)= - (2\eps)^{1/2}$,
$f\circ \varphi_1 (\alpha) = f(-\delta)= - (\eps+\delta)^{1/2}$,
so that
\[
\frac{\alpha}{\beta} \ge
\frac{(2\eps)^{1/2}-(\eps+\delta)^{1/2}}{(2\eps)^{1/2}+(\eps+\delta)^{1/2}}
=
\frac{\eps-\delta}{((2\eps)^{1/2}+(\eps+\delta)^{1/2})^2}
\ge
\frac{\eps-\delta}{8\eps},
\]
so that $\alpha \gtrsim \frac{(\eps-\delta)\beta}{\eps}\asymp(\eps-\delta) \eps^{\frac1{2\mu}-1}$.
\smallskip

{\it Upper bound for $\ell_{G_\mu}^{(2)}$ when $\frac12 < \mu \le 1$.}

Let $q:=(-\delta, \eps-\delta)$.
Considering the disc $\zeta \mapsto (\zeta - \eps, \zeta)$,
which maps to $G_1$ since $-\eps + \Re \zeta < |\zeta|$,
we see that $\ell_{G_\mu}(p_\eps, q)\le \eps-\delta $;
and considering $\zeta \mapsto (-\delta, \zeta)$, that
$\ell_{G_\mu}(p_\delta, q)\le \eps-\delta $.

\smallskip

{\it Upper bound for $\ell_{G_\mu}^{(2)}$ when $1< \mu $.}

To bound
$\ell_{G_\mu}^{(2)}(p_\delta, p_\eps)$ from above,
we need to see that the disc
$\zeta \mapsto (- \eps+ C \eps^{1-\frac1\mu}\zeta , \zeta)$
maps to $G_\mu$ when $C>0$ is small enough.
It is enough to check that $F_\mu(x) := x^\mu - C \eps^{1-\frac1\mu}x+\eps >0$ for $0<x<1$. The function $F_\mu$
attains a minimum for
$x=x_\mu:= \left(\frac{C}{\mu}\right)^{\frac{1}{\mu-1}}\eps^{\frac{1}{\mu}}$,
and $F_\mu(x_\mu)= \eps(1+ C^\frac{\mu}{\mu-1}(\mu^\frac{\mu}{1-\mu}-\mu^\frac{1}{1-\mu}))>0$ for $C$ small enough.

Then we set $q:=(-\delta, \frac{\eps-\delta}{C \eps^{1-\frac1\mu}})$,
and the disc defined above shows that
$\ell_{G_\mu}(p_\eps, q)\le \frac{\eps-\delta}{C \eps^{1-\frac1\mu}} $,
while as before it is plain that
$\ell_{G_\mu}(p_\delta, q)\le \frac{\eps-\delta}{C \eps^{1-\frac1\mu}}$.
\smallskip

{\it Lower bound for $\ell_{\tilde G_\mu}^{(2)}$ when $\frac12 \le \mu \le 1$.}

This is covered by the remark at the outset of the proof that proves that $\ell_{\tilde G_\mu}^{(2)}(p_\delta, p_\eps) \gtrsim \eps-\delta$
in all cases for this range of $\mu$.

\smallskip

{\it Lower bound for $\ell_{\tilde G_\mu}^{(2)}$ when $1< \mu $.}

For the lower bound, 
it is enough
to bound $k_{\tilde G_\mu}(p_\delta, p_\eps)$ from below.  Since we will obtain a quantity of 
the same type as the upper bound, all the $\ell^{(m)}_{\tilde G_\mu}$ must have the same order of magnitude.

To simplify notation, we write this proof for $n=2$.

Using \eqref{kobadef}, we consider a near-extremal curve
$\gamma:[0;1]\longrightarrow \tilde G_\mu$ such that $\gamma(0)=p_\eps$
and $\gamma(1)=p_\delta$.  Denote by $\delta_{\tilde G_\mu}(z):=
\inf\{\|z-w\|, w\in \C^2\setminus \tilde G_\mu\}$, and by $\pi(z)$
a point in $\partial \tilde G_\mu$ such that $\|z-\pi(z)\|=\delta_{\tilde G_\mu}(z)$.
This is a slight abuse of notation, but for $\mu\ge 1$,
 in a neighborhood
of $(0,0)$, where all our points will lie, the complement
of $\tilde G_\mu$ is convex and so $\pi(z)$ is uniquely defined.
Finally for $z$ close enough to $\partial \tilde G_\mu$,
let $\nu_z$ stand for the inner normal vector to $\partial \tilde G_\mu$
at $\pi(z)$.

Since
$k_{\tilde G_\mu}(z,w)\gtrsim \|z-w\| $, observe that the upper bound tells
us that we may assume $\|\gamma(t)-p_\eps\| \lesssim \eps^{\frac1\mu}$
for $t \in [0;1]$. It follows that $\|\gamma(t)\| \lesssim \eps^{\frac1\mu}$,
thus $\delta_{\tilde G_\mu}(\gamma(t))\lesssim \eps^{\frac1\mu}$,
thus $\|\pi(\gamma(t))\|\lesssim \eps^{\frac1\mu}$. Finally
the equation of $\partial \tilde G_\mu$ implies $0\le \Re \pi(\gamma(t))_1
\lesssim \eps$.

Computing the derivative of $|z_2|^\mu + |\Im z_1|^\mu$, we also see
that
\begin{equation}
\label{nu2}
|\Im(\nu_{\gamma(t)})_1| + |(\nu_{\gamma(t)})_2| \lesssim \|\pi(\gamma(t))\|^{\mu-1}
\lesssim  \eps^{1-\frac1\mu}.
\end{equation}
In particular, since
$\nu_{\gamma(t)}$ is a unit vector, for $\eps$ small
enough, $|\Re (\nu_{\gamma(t)})_1|\ge \frac{1}{2}$.

Using the Intermediate Value Theorem,
 we may consider  the restriction of $\gamma$
to some interval $[a;b]\subset [0;1]$ such that
$\Re \gamma(a)_1 =-\eps$, $\Re \gamma(b)_1 =-\delta$,
and $-\eps\le \Re \gamma(t)_1 \le -\delta$, for $a\le t \le b$.
Of course
\[
k_{\tilde G_\mu}(p_\delta, p_\eps) \ge \int_a^b\kappa_{\tilde G_\mu}(\gamma(t);\gamma'(t))dt.
\]
From now on we reparametrize this restriction to be defined on $[0;1]$
and denote it again by $\gamma$.
Since $-\eps \le \Re \gamma(t)_1$, $\delta_{\tilde G_\mu}(\gamma(t))\lesssim \eps$.

Consider the positive
measure $\theta$ on $[0;1]$ defined by $d\theta(t)= |\Re \gamma'(t)_1| dt$,
so that $\theta([0;1])\ge \eps-\delta$,
and the set
\[
A:= \left\{ t\in [0;1]:
\left| \gamma'(t)_2 (\nu_{\gamma(t)})_2 \right| + \left| \Im \gamma'(t)_1 \Im (\nu_{\gamma(t)})_1 \right|
\le \frac{1}{2} \left| \Re \gamma'(t)_1 \Re (\nu_{\gamma(t)})_1 \right|
\right\} .
\]

{\bf Case 1.} Suppose $\theta(A) \ge \frac{1}{2}\theta([0;1])$.

We can deduce from \cite[Proposition 2.3]{Fu2} that
\[
\kappa_{\tilde G_\mu} (z; X) \gtrsim \frac{|\langle X, \nu_z \rangle|}{\delta_{\tilde G_\mu}(z)^{1-\frac1\mu}} + \|X\|.
\]
For $t\in A$,
\[
|\langle \gamma'(t), \nu_{\gamma(t)} \rangle|
\ge \frac{1}{2} \left| \Re \gamma'(t)_1 \Re (\nu_{\gamma(t)})_1 \right|
\ge \frac{1}{4} \left| \Re \gamma'(t)_1 \right| ,\]
so, since $\delta_{\tilde G_\mu}(\gamma(t))^{1-\frac1\mu}\lesssim \eps^{1-\frac1\mu}$,
\[
k_{\tilde G_\mu}(p_\delta, p_\eps) \gtrsim \eps^{\frac1\mu -1} \int_A | \Re \gamma'(t)_1| dt \gtrsim \eps^{\frac1\mu -1} (\eps-\delta).
\]

{\bf Case 2.} Suppose $\theta(A) < \frac{1}{2}\theta([0;1])$.

Then \eqref{nu2} implies that for $t\in [0;1]\setminus A$,
\[
\|\gamma'(t) \|  \eps^{1- \frac1\mu } \gtrsim
\left| \gamma'(t)_2 (\nu_{\gamma(t)})_2 \right| + \left| \Im \gamma'(t)_1 \Im (\nu_{\gamma(t)})_1 \right|
\ge \frac{1}{2} \left|\Re \gamma'(t)_1 \Re (\nu_{\gamma(t)})_1 \right|
\ge \frac{1}{4} \left| \Re \gamma'(t)_1  \right| ,
\]
so
\[
k_{\tilde G_\mu}(p_\delta, p_\eps) \gtrsim
\int_{[0;1]\setminus A} \|\gamma'(t) \| dt
\gtrsim
\eps^{\frac1\mu -1} \int_{[0;1]\setminus A}\left| \Re \gamma'(t)_1  \right| dt
\gtrsim \eps^{\frac1\mu -1}  (\eps-\delta).
\]
\end{proof}

\begin{rem}
\label{avoid}
If $D:= \B^n \setminus L$, where $L$ is closed, $L\cap\B^n\neq\emptyset$ and $H_{2n-2}(L)=0$, with $H_{2n-2}$ being the Hausdorff measure in real
dimension $2n-2$, then it is an immediate consequence of \cite[Theorem 3.4.2]{JP}
that  for any $z,w \in D$, $\ell_D (z,w)=\ell_{\B^2}(z,w)$,
and as a consequence $l_D$ verifies the triangle inequality. However $D$ is not pseudoconvex.
\end{rem}

On the other hand, if $D:= \B^2 \setminus (\C\times \{0\})$, or  $D:=\D \times \D^*$, then $D$ is pseudoconvex
(and we do not know whether a QTI holds there).  

\section{More estimates for the growth of metrics}
\label{more}


\subsection{Estimates for the Sibony metric.}

Recall the definition of the Sibony metric of a domain $D$: for $p\in D$, $X\in \C^n$,
\[
S_D(p,X):=\sup\left\{ \partial\bar\partial u(p)(X, \bar X)^{1/2}
:=\left(\sum_{i,j=1}^n\frac{\partial^2 u}{\partial z_i\partial \bar z_j}(p) X_i\bar X_j\right)^{1/2}: u\in A(p,D)\right\},
\]
where $\D$ denotes a unit disc in $\C$, and $A(p,D)$ is the set of (plurisubharmonic) functions on $D$ such that  
$u(p)=0$, $u$ is $\mathcal C^2$ near $p$, $\log u$ is plurisubharmonic on $D$, and $0\le u\le 1$ on $D$. 

The Finsler metric $S_D$ satisfies the triangle inequality (see e.g. \cite[Lemma 2]{FL}).

\begin{prop}
\label{sib}
Let $\mu>1$ and $G_\mu$ be as in \eqref{defG}.
In general, for $z\in \C^n$, we write $z=(z_1,z') \in \C \times \C^{n-1}$.
Then $S_{G_\mu}(p_\eps; X)
\gtrsim \eps^{\frac1\mu -1} |X_1| + \|X\|$.
\end{prop}

\begin{proof}
We follow the method of proof of \cite[Proposition 3]{FL}.
Using the fact that $\{-\eps\}\times \D^{n-1} \subset G_\mu 
\subset \D^n$, it is easy to see that $S_D (p_\eps; (0,X') \asymp \|X'\|$.  So, using the triangle inequality, 
it will be enough to prove $S_{G_\mu}(p_\eps; (1,0,\dots,0)) \gtrsim \eps^{\frac1\mu -1}$ to imply the result.           

For $\zeta\in\C$, let $f(\zeta)= \eps^{2/\mu} \left|\frac{\zeta+\eps}{\zeta-\eps} \right|^2$.
An elementary computation shows that $\frac{\partial^2f}{\partial \zeta \partial \bar \zeta}(-\eps) = \frac14 (\eps^{\frac1\mu -1})^2$. 

Choose $c_1$ a small enough constant so that $z\in G_\mu$ and $\|z'\|\le c_1 \eps^{1/\mu}$
implies $\Re z_1 \le \eps/2$. Let $\alpha>0$. We will choose $L,L'>0$ so that $e^u$ is a candidate for the computation of $S_{G_\mu}(p_\eps; (1,0,\dots,0))$, where
\[
u(z,w):=
\bcases
\max\set{\log\br{f(z_1)+\|z'\|^2}, \log\br{L \|z'\|^{2+\alpha}}}-L',& \|z'\|\le c_1 \eps^{1/\mu}\\
\log\br{L\|z'\|^{2+\alpha}}-L', & \|z'\|\ge \frac{c_1}2 \eps^{1/\mu}.
\ecases
\]
Clearly for $z$ close enough to $p_\eps$ this coincides with $\log\br{f(z-1)+\|z'\|^2}$, 
so $\left( \frac{\partial^2 (e^u)}{\partial z_1 \partial \bar z_1}(p_\eps)\right)^{1/2} \gtrsim \eps^{\frac1\mu -1}$.  We need to see that the function is well defined and plurisubharmonic,
which will be achieved if $\log\br{L\|z'\|^{2+\alpha}}-L'\ge \log\br{f(z_1)+\|z'\|^2}$ in the 
``collar'' region $ \frac{c_1}2 \eps^{1/\mu}\le \|z'\|\le c_1 \eps^{1/\mu}$.

Notice that in this region $\Re z_1 \le \eps/2$, so $f(z_1) \le c_2 \eps^{2/\mu}$,
while $\|z'\|^{2+\alpha}\ge c_3 \eps^{(2+\alpha)/\mu}$. So it is enough to take 
$L\ge c_2 c_3^{-1} \eps^{-\alpha/\mu}$. Finally, taking $L'$ large enough, we can ensure
that $u\le 0$ on $G_\mu$. 
\end{proof}

\subsection{Estimates near a $\mathcal C^{1,\gamma}$-smooth point.}

The following is a consequence of our estimates on model domains. 

\begin{prop}
\label{onegamma}
Let $D \subset \C^n$ be a bounded $\mathcal C^{1,\gamma}$-smooth  domain, for $0<\gamma \le 1$, and
for $p\in \partial D$, let $p_t := p+t n_p$, where $n_p$ is the inner normal vector at $p$. Then for 
$0< \delta < \eps$ small enough, $\ell_D(p_\eps,p_\delta) \gtrsim (\eps-\delta)\eps^{\frac1{2(1+\gamma)}-1}
\asymp \eps^{\frac1{2(1+\gamma)}} - \delta^{\frac1{2(1+\gamma)}}$.
\end{prop}

Note that $t=|p_t-p|$ and that, even though $p$ need not be the closest point to $p_t$, $\lim_{t\to0, t>0} \frac{|p_t-p|}{\delta_D(p_t)} =1$.

\begin{proof}
Let $\mu:=1+\gamma$.
We can follow the beginning of the proof of Proposition \ref{C1}, this time using $\Phi (\Im z_1, z') \le C(|\Im z_1|^\mu+|z'|^\mu)$,
which follows from $\nabla \Phi (0)=0$ and the regularity hypothesis. So we can reduce ourselves to studying a lower bound for 
$l_{\tilde G_\mu}$, and Proposition \ref{modelest} gives the answer.
\end{proof}

Conversely, in a pseudoconcave situation with the same smoothness we can give an upper estimates for $\ell^{(2)}_D(p_\eps,p_\delta)$
and $ \kappa^{(2)}_D (p_\eps;n_p)$. 

\begin{prop}
\label{onegammatwo}
Let $D= \{\rho<0\} \subset \C^n$  be a bounded $\mathcal C^{m,\gamma}$-smooth  domain, for $0\le \gamma <1$, $m\in \N^*$, and
for $p\in \partial D$, let $p_t := p+t n_p$, where $n_p$ is the inner normal vector at $p$. 
Assume that there is a nonzero vector $v \in T_p^{\C}\partial D$ (the complex tangent space to $\partial D$ at $p$) and $c>0$
such that $\rho (p + \lambda v) \le - c |\lambda|^{m+\gamma}$.

Then 
\begin{enumerate}
\item
for $0< \delta < \eps$ small enough, $\ell_D(p_\eps,p_\delta) \lesssim (\eps-\delta)\eps^{\frac1{(m+\gamma)}-1}
\asymp \eps^{\frac1{(m+\gamma)}} - \delta^{\frac1{(m+\gamma)}}$;
\item
for $0< \eps$ small enough, $ \kappa^{(2)}_D (p_\eps;n_p) \lesssim \eps^{\frac1{(m+\gamma)}-1}$.
\end{enumerate}
\end{prop}

\begin{proof}
Let $\mu:=m+\gamma$. After an affine change of variables, we may assume that $0\in \partial D$ and that $G^-_\mu \times \{0\} \subset D$,
where $G^-_\mu := \{z\in \D^2: \Re z_1 < -| \Im z_1| +  |z_2|^\mu \}$, with $p_t=(-t,0)$, $0<t<1$.
By inclusion, the proof will reduce to the next Lemma.
\end{proof}

\begin{lem}
\label{modelminus}
For $\mu\ge 1$ and $\eps, \delta$ small enough, (1) $\ell_{G^-_\mu}^{(2)}(p_\eps,p_\delta) \lesssim (\eps-\delta)\eps^{\frac1{\mu}-1}$,
and (2) for $0< \eps$ small enough, $ \kappa^{(2)}_{G^-_\mu} (p_\eps;n_p) \lesssim \eps^{\frac1{\mu}-1}$.
\end{lem}

\begin{proof}
The proof of (1) is analogous to that of Proposition \ref{upbd}. 

Let $q_\delta:= \left(-\delta, \frac{\eps-\delta}{\eps^{1-1/\mu} \sqrt2}\right)$. 
Clearly $\ell_{G^-_\mu}(q_\delta,p_\delta)\le \frac{\eps-\delta}{\eps^{1-1/\mu} \sqrt2}$.

Let $\varphi(\zeta):= (-\eps+\eps^{1-1/\mu}\frac\zeta{\sqrt 2}, \zeta)$. For small $\eps$, $\varphi(\D)\subset \D^2$. 
On the other hand, $\varphi(0)= p_\eps$, $\varphi\left(\frac{\eps-\delta}{\eps^{1-1/\mu} \sqrt2}\right) = q_\delta$. It remains
to see that $\varphi(\D)\subset G^-_\mu$ by checking the inequation.  Let $\zeta:= x+iy$. It is enough to see that 
for any $x,y$,
\[
\eps^{1-1/\mu} (|x|+|y|) 2^{-1/2} \le \eps + (|x|+|y|)^\mu 2^{-\mu/2} \le \eps + (x^2+y^2)^{\mu/2}.
\]
The second inequality is elementary.  For the first one, let $t=|x|+|y|$ and notice that for $0\le t \le \eps^{1/\mu} \sqrt2$, $\eps^{1-1/\mu} t^\mu 2^{-1/2}\le \eps$; while for $\eps^{1/\mu} \sqrt2\le t$, $t^{\mu-1} 2^{-\mu/2} \ge \eps^{1-1/\mu} 2^{-1/2}$.

To prove (2), write $(1,0)= (1, \eps^{\frac1\mu-1} ) - \eps^{\frac1\mu-1} (0,1)$. Then 
$ \eps^{1-1/\mu} \sqrt2 \varphi'(0) = (1, \eps^{\frac1\mu-1} )$
so $ \kappa_{G^-_\mu} (p_\eps; (1, \eps^{\frac1\mu-1} ) ) \le\frac{\eps^{\frac1{\mu}-1}}{\sqrt 2}$; and since
clearly $\kappa_{G^-_\mu} (p_\eps; (0,1) ) =1$, the result follows.
\end{proof}

\section{Appendix}
\label{appen}

\begin{proof*}{\it Proof of Proposition \ref{cont}.}
We give the proof for $l_D$, the same arguments
work for $\ell_D$, {\it mutatis mutandis.}

(a) Note that $l_D$ is an upper semicontinuous function
(see e.g. \cite[Proposition 3.1.14]{JP}). Since $l_D^{(m)}$ is an infimum of such functions, it is
upper semicontinuous too. The uniformity easily follows from the fact that $(l_D^{(m)})_{m\in\N}$
is a pointwise decreasing sequence which is bounded from below.
\smallskip

(b) For $m=1$ the results follows by tautness of $D$ (see e.g. \cite[Propositions 3.2.7 and 3.2.9]{JP}).
Further, we will consider only the case $m=2$ to simplify notations. The same proof works for $m\ge 3$
with more intermediate points.

Let $a,b\in D.$ We may take take a sequence of points $c_k\in D$ such that
$l_D(a,c_k)+l_D(c_k,b)\to l_D^{(2)}(a,b).$ Since $D$ is taut, we have the the following.
\smallskip

{\bf Claim.} Given $L>0$ and $U\Subset D$, there exists $K\Subset D$ such that
for any $a'\in U$ and any $c\in D$ for which $l_D(a',c)\le L$, one has that $c\in K.$
\smallskip

Note that the analogous fact, but for a single point $a$ instead of
a relatively compact set $U$, follows from \cite[Proposition 3.2.1]{JP}.
The same proof with obvious modifications implies the claim.

It follows that, passing to a subsequence if necessary, we may assume that $c_k\to c\in D.$
Then $l_D(a,c)+l_D(c,b)=l_D^{(2)}(a,b)$ by continuity of $l_D.$ Since $l_D$
admits extremal discs, so does $l_D^{(2)}$.

Suppose now that $l_D^{(2)}$ is not continuous at $(a,b)\in D\times D.$ Since it is upper semicontinuous,
we can find $\delta>0$ and sequences $a_k \to a$ and $b_k \to b$ such that
$\lim_{k\to \infty}l_D^{(2)}(a_k,b_k) \le l_D^{(2)}(a,b) -3\delta$. Then for each $k$ large enough
we may choose a point $c_k\in D$ such that
$$l_D(a_k,c_k) +  l_D(c_k,b_k) \le  l_D^{(2)}(a_k,b_k) + \delta \le  l_D^{(2)}(a,b)- \delta.$$
By the claim above, we may assume that $c_k\to c\in D$. Since $l_D$ is continuous, we get the
contradiction $l_D^{(2)}(a,b) \le l_D(a,c)+l_D(c,b) \le l_D^{(2)}(a,b) -\delta.$

\smallskip
An analogous proof can be written for $\kappa^{(m)}_D$.
\end{proof*}

\smallskip

\noindent{\bf Acknowledgement.} The authors would like to express their heartfelt thanks to Peter Pflug,
whose cogent remarks much improved the manuscript.

{}

\end{document}